\documentclass[10pt,draftcls,onecolumn]{IEEEtran}

\usepackage{amsmath,amsfonts,balance}
\usepackage{graphicx}
\usepackage{hyperref}
\newcommand{\disp}{\displaystyle}

\newcommand{\bed}{\begin{displaymath}}
\newcommand{\eed}{\end{displaymath}}
\newcommand{\bea}{\bed\begin{array}{rl}}
\newcommand{\eea}{\end{array}\eed}
\newcommand{\ad}{&\!\!\!\disp}
\newcommand{\aad}{&\disp}
\newcommand{\barray}{\begin{array}{ll}}
\newcommand{\earray}{\end{array}}

\newif\ifblog
\newif\iftex
\blogfalse
\textrue

\newcommand{\mv}{Mean-Field }

\def\em{\it}
\def\emph#1{\textit{#1}}

\def\E{{\mathbb E}}

\newcommand{\wdt}{\widetilde}

\newtheorem{theorem}{Theorem}

\newtheorem{definition}[theorem]{Definition}

\newtheorem{proposition}[theorem]{Proposition}
\newtheorem{example}{Example}

\newenvironment{proof}{\noindent {\sc Proof:}}{$\Box$} 

\begin{document}

\title{Solving A Class of Mean-Field LQG Problems}

\author{Yun Li, Qingshuo Song, Fuke Wu,  George Yin
\thanks{Y. Li is with Huazhong University of Science and Technology, (li\_yun@hust.edu.cn).
The research of this author was supported in part by the National Natural Science Foundation of China (Grant No.61873320).}
\thanks{Q. Song is with
Worcester Polytechnic Institute, and
City University of Hong Kong,
(qsong@wpi.edu). The research of this author was
supported in part  by the RGC of Hong Kong CityU (11201518).}
\thanks{F. Wu is with
Huazhong University of Science and Technology,
(wufuke@hust.edu.cn).
The research of this author was supported in part by the National Natural Science Foundation of China (Grant No.61873320).}
\thanks{G. Yin is
with the
University of Connecticut, Storrs, CT 06269-1009, USA, (gyin@uconn.edu). The research of this author was
supported in part by the Air Force Office of Scientific Research under grant FA9550-18-1-0268.}
\thanks{The authors are grateful to
Professor Peter Caines for several stimulating
discussions on this work
during his visit to Hong Kong City University
in summer 2019.}
}

\markboth{IEEE Transactions on Automatic Control}{Li, Song, Wu, Yin, Mean-Field LQG Problems}

\maketitle

\begin{abstract}
In this work,
we  study a class of mean-field linear quadratic Gaussian (LQG) problems. Under suitable conditions,  explicit solutions of the distribution-dependent optimal control problems are obtained.
Riccati systems are derived by directly solving the associated master equations.
Some extensions on controls with partial observations are also considered.
\end{abstract}


\IEEEpeerreviewmaketitle

\section{Introduction}

We consider a one-dimensional LQG problem.
Suppose the controlled process
$X_t \in {\mathbb R}$ is the solution of
a stochastic differential equation
\begin{equation}
\label{eq:Xt}
dX_t =(A_t X_t + B_t u_t) dt + \sigma_t dW_t,
\end{equation}
where $A_t$, $B_t$, and $\sigma_t$ are suitable functions of $t$, $W_t$ is a standard real-valued Brownian motion, and $u_t$ is the control.
The objective is to minimize an expected cost function of the form
$$J(x,u )= \E_x \Big[\int^T_0 (R_t X^2_t+ Q_t u^2_t) dt\Big]
+ \hat g(X_T),$$
where $R_tX^2_t + Q_t u^2_t$ is the running cost rate
and $\hat g(X_T)$ is the terminal cost.

If the terminal cost is $\hat g(X_T) = \mathbb E[X_T^2]$, it is
 the classical LQG problem; see, for example,
Fleming and Rishel \cite{FR75} and Yong and Zhou \cite{YZ99}, among others.
There is a vast
literature for LQG control problems
 under complete observations as well as partial observations; see for example,
\cite{FR75, DMS12, HR92,YZ99} and
related works in \cite{Cai18, HCM12, NH12, NNY19}, among others.
It is now standard that the associated  Hamilton-Jacobi-Bellman (HJB) equations can be solved by the associated
Riccati equations
provided if the cost function is quadratic in the states and controls.

In this work, we study the control problem with terminal cost
given by a function not of the state but
 the distribution $\mu_T$ of the terminal state $X_T$. For instance,
consider
$\hat g(X_T) = g(\mu_T) =  (\mathbb E[X_T])^2$. Then it does not belong to
traditional setup of LQG problem. As noted in \cite{Yon13} and \cite{BKM16},
this problem belongs to the class of time-inconsistent
 control problems.
 Indeed, in such a problem,  the
 dynamic programming principle (DPP) is not applicable.

An extensive literature is devoted to time-inconsistent control problems; see \cite{BMZ14, Yon15,Yon12, BKM16, ZY03} and the references therein.
It is worth 
mentioning 
because of no
time-consistent optimal controls,
the focus in the above references is to find
``locally optimal" time-consistent controls,
which is referred to as ``equilibrium solution".

We emphasize that the optimal solutions are strictly different from the equilibrium solution discussed in the aforementioned references.
For the optimal solution,
\cite{Yon13} provides Riccati system based on decoupling technique for FBSDE; see also Example 1.2 of \cite{Yon15}, Section 6.7 of \cite{CD18I}, \cite{Pha16},
and \cite{Zha19}.


In contract to
the aforementioned works, our
 aim is to obtain explicit solutions by
solving its associated master equation directly
in Section \ref{s:lqg}. The solution will provide us with insight on the dependence of the solution on the associated distribution.
The key is to identify the time-inconsistent
problem as a LQ control problem in a suitable sense, where linear and quadratic structure
is referred to the functions with domains being suitable measure spaces. Similar to the
approach of traditional LQG, we also guess the solution of the
master equation as a quadratic function of the associate measure.
This approach successfully reduces originally infinite dimensional master equation to a finite dimensional Riccati system after explicit computations
using L-derivatives; see
Section \ref{sec:form} and \cite{CDLL19, CD18I} for a brief introduction of L-derivatives. Using our new approach,
Example \ref{ex:02} in this paper 
recovers Example 1.2 of \cite{Yon15}. As a result, the optimal
trajectory is a Gaussian process, which justifies the underlying LQ
problem being linear quadratic Gaussian.

Section \ref{s:cn}  is concerned
with an extension of mean-field LQG
in which the system is
only partially observable.
The optimal control can be obtained by a separation principle
to covert the partially observed system to a fully observed one.
Finally, we 
conclude the paper with a brief discussion in Section \ref{s:sum}.

\section{Preliminaries}\label{sec:form}

\subsection{Polynomials and Derivatives on Measure Space}
\label{s:derivative}

Suppose $\mu$ is a distribution
on Borel sets $\mathcal B(\mathbb R)$ and
$f: \mathbb R\mapsto \mathbb R$ is a real-valued function.
We write $$\langle f, \mu \rangle := \int_{\mathbb R} f(x) \mu(dx),$$
if the integral exists.
We denote by $$[\mu]_{m} := \langle x^{m}, \mu\rangle$$ the $m$th moment for any $m\ge 1$.
If a distribution $\mu$ has a finite $m$th moment $[\mu]_{m}$,
then we write it as $\mu \in \mathcal P_{m}$.
For instance,  for any $x\in \mathbb R$,
a Dirac measure $\delta_{x}$
belongs to $\mathcal P_{m}$ for any $m\ge 1$, since
$[\delta_{x}]_{m} =  x^{m}$ holds.

Polynomials on $\mathcal P_2$ are defined as a linear combination of the monomials defined in this below.
\begin{enumerate}
\item
A 1-monomial is given by a function in the form of
$$f(\mu) = \langle\phi, \mu\rangle$$
for some appropriate
function $\phi: \mathbb R\mapsto \mathbb R$.
\item An  $n$-monomial is a product of $n$ many 1-monomials,
$$f(\mu) = \Pi_{i=1}^n \langle \phi_i, \mu \rangle,$$
for some coefficients  $\phi_i$.
\end{enumerate}

We
use
a notion of L-derivative on the functions of probability measures
in a lifted space.
We summarize below a few useful results to be used in this paper.
\begin{enumerate}
\item
The derivative of 1-monomial becomes $\mu$-invariant,
$$\partial_\mu \langle \phi, \mu \rangle = \phi'(x).$$
\item Chain rule and product rules can be used as usual, which  yields that the derivative of
$n$-monomial becomes $(n-1)$-monomial. For instance, we have
$$\partial_\mu ([\mu]_m)^n = n [\mu]_m^{n-1} m x^{m-1}.$$
\end{enumerate}
Note that the  notion of L-derivative $\partial_\mu f$ is taken from \cite{CD18I}, which  is equivalent to the intrinsic derivative $D_\mu f$ introduced by
\cite{CDLL19}, that is,
$$\partial_\mu f(\mu, x) = D_\mu f (\mu, x) = \partial_x \frac{\delta f}{\delta \mu}(\mu, x).
$$

\subsection{Verification Theorem}
Let $(\Omega, \mathcal F, \mathbb P, \mathbb F)$ be a complete filtered probability space satisfying the usual conditions,
where
${\mathbb F}= (\mathcal F_{t})_{t\ge 0}$ is the filtration
on which there exists an $\mathbb F$-adapted Brownian motion $W$.
Given a controlled SDE
\begin{equation}
 \label{eq:con1}
X_{t} = x + \int_{0}^{t} b( s, X_{s}, u_{s}) ds + \int_{0}^{t} \sigma_{s} dW_{s},
\end{equation}
we denote by $\mu_{t}$  the probability law of $X_{t}$ and consider the cost function
\begin{equation}
 \label{eq:con2}
 J(u
 ) =
\mathbb E
 \Big[ \int_{0}^{T} \ell( t, X_{t}, u_{t}) dt \Big] + g(\mu_{T}).
\end{equation}
In the above, $u $ is an $\mathcal F_{t}$ progressively measurable control process,  $\ell(\cdot, \cdot, \cdot)$ is the running cost function, and $g(\cdot)$ is the terminal cost.
Our objective is to minimize the cost function $J$ over an admissible control space $\mathcal U$, i.e.,
\begin{equation}
 \label{eq:con3}
V^{*} =
J(u^{*}
) \le
J(u ),
\ \forall u  \in \mathcal U.
\end{equation}

\begin{definition}\label{def:fb}
 A random process
$u: [0, T]\times \Omega \mapsto \mathbb R$ is  said to be
admissible
if $u$ together with $(X, J)$ satisfies
\eqref{eq:con1}-\eqref{eq:con2}  and
\begin{equation}
 \label{eq:con31}
u_{t} = a (t, \mu_{t}, X_{t}) \hbox{ for all } t\in [0, T]
\end{equation}
for some controller $a$ in the feedback form of $(t, \mu_t, X_t)$.
The collection of all such admissible controls is denoted by
 $\mathcal U$.
\end{definition}

Since the terminal cost is a function of a measure,
we lift the optimal value $V^*$ to a value function
of the form $V(t, \mu)$ such that
$V^* = V(0, \delta_x)$ accordingly.
The verification theorem says that
under sufficient regularity,
the value function $V(t, \mu)$
solves the following master equation
 \begin{equation}
 \label{eq:hjb02}
 \barray \ad
\inf_{a \in \mathcal M(\mathbb R)}\langle
H(t, \cdot, \mu, v, a(\cdot)), \mu
\rangle
\\
\aad \ \quad \quad \quad\quad
+ \frac 1 2 \sigma^{2}_t
\langle
\partial_{x\mu} v (t, \mu, \cdot), \mu
\rangle
+ \partial_{t} v (t, \mu) = 0,
\earray
\end{equation}
with the terminal condition
\begin{equation}
 \label{eq:tc02}
 v (T, \mu) = g(\mu),
\end{equation}
where $\mathcal M(D)$ is the collections of all real-valued measurable mappings on a  metric space $D$, and $H$ is given as
$$H(t, x, \mu, v, a) =
b(t, x, a)
\partial_{\mu} v (t, \mu, x) +
\ell(t, x, a).$$


Throughout  the rest of the paper, we use the convention
$f(t, \mu)(x) = f(t, \mu, x)$.
To proceed, we say a function
$f: [0, T] \times \mathcal P_{2} \mapsto \mathbb R$ is partial
$\mathcal C^{1,2}$
 if there exists continuous derivatives
$\partial_t f, \partial_\mu f, \partial_{x\mu} f: [0, T] \times \mathcal P_{2} \times \mathbb R \mapsto \mathbb R$. For convenience, we denote by $\mathcal C_I$
all partial $\mathcal C^{1,2}$ functions $f$ satisfying a growth condition
$\langle |\partial_{x \mu} f|^2, \mu\rangle \le C (1 +  [\mu]_2^m)$
for some $C, m>0$.
Recalling the chain rule \cite[Proposition 5.102]{CD18I},
a function $f \in \mathcal C_I$  satisfies
$$
\begin{array}
{ll}
f(t, \mu_{t}) =
f(0, \mu_{0}) +
\int_{0}^{t} \mathbb E[\partial_{\mu} f(s, \mu_{s}, X_{s}) b(s, X_{s}, u_{s})] ds\\
\hspace{.4in} + \frac 1 2 \int_{0}^{t} \mathbb E
[\sigma_{s}^{2} \partial_{x\mu} f(s, \mu_{s}, X_{s})] ds +
\int_{0}^{t} \partial_{t} f(s, \mu_{s}) ds.
\end{array}
$$

\begin{proposition}
 \label{p:veri01}
 Let $b$ and $\ell$ be Lipstchitz continuous in $(t, x)$. Suppose there exists a solution
 $v \in \mathcal C_I$
 of the master equation \eqref{eq:hjb02}-\eqref{eq:tc02} and
 a feedback form
 $a^*: (0, T) \times \mathcal P_{2} \times \mathbb R \mapsto \mathbb R$ satisfying the optimality condition
\begin{equation}
 \label{eq:opt01}
H(t, x, \mu, v, a^{*}(t, \mu, x)) =
 \inf_{a \in \mathbb R}
H(t, x, \mu, v, a),
\end{equation}
for all $(t, \mu, x) \in (0, T) \times \mathcal P_{2} \times \mathbb R$.
In addition, if there exists an optimal pair
$(X^{*}, u
^{*})$
of state trajectory and admissible control
satisfying
$$
u^{*}_{t} = a^{*}(t, \mu^{*}_{t}, X^{*}_{t}),
$$
then the optimal value is
$$V^{*} = v (0, \delta_{x}).$$
\end{proposition}

\begin{proof} 
Applying the chain rule to the solution $v$ of the master equation,
for any control $u
\in \mathcal U$, we have
 $$
\begin{array}
 {lll}
v(t, \mu_{t}) &=
 &\disp g(\mu_{T}) - \int_{t}^{T} \partial_{t} v (s, \mu_{s}) ds \\
& &\disp - \int_{t}^{T}
\mathbb E[\partial_{\mu} v (s, \mu_{s}, X_{s}) b(s, X_{s}, u_{s})] ds
 \\& &\disp
  - \frac 1 2 \int_{t}^{T} \mathbb \sigma_{s}^{2}
  \mathbb E[\partial_{x\mu} v (s, \mu_{s}, X_{s})] ds
\end{array}
$$
Since $u\in \mathcal U$ and $v$ solves \eqref{eq:hjb02},
there exists feedback form
$u_{t} = a(t, \mu_{t}, X_{t})$ and  we can write$$
\barray \ad
\{
\langle
H(s, \cdot, \mu_{s}, v, a(s, \mu_{s}, \cdot)), \ \mu_{s}
\rangle \} +
\\
\aad \ \quad \quad \quad\quad
\frac 1 2 \sigma^{2}_s \langle
\partial_{x\mu} v (s, \mu_{s}, \cdot), \mu_{s}
\rangle
+ \partial_{t} v (s, \mu_{s}) \ge 0,
\earray
$$
Therefore, with the definition of $H(\cdot)$, we obtain
$$
\begin{array}
 {lll}
\disp
\mathbb E[\partial_{\mu} v (s, \mu_{s}, X_{s}) b(s, X_{s}, u_{s})] +
 \\
\disp \quad
 \frac 1 2  \mathbb \sigma_{s}^{2}
  \mathbb E[\partial_{x\mu} v (s, \mu_{s}, X_{s})]
 +\partial_{t} v (s, \mu_{s})\ge
   - \mathbb E [\ell(s, X_{s}, u_{s})].
\end{array}
$$
This implies that
$$v (0, \mu_{0}) \le g(\mu_{T}) +\int_{0}^{T} \mathbb E [\ell( t, X_{t}, u_{t})] dt
= J(u
)
$$
for any control $u  \in \mathcal U$ and initial distribution $\mu_{0}$.
The other direction $V^{*}  = J(u^{*}) \le J(u)$ is straightforward.
\end{proof}

The verification theorem has been studied in various
forms for
McKean-Vlasov control problems, for instance,
Proposition 6.32 of \cite{CD18I}.
Proposition \ref{p:veri01} is tailor-made for our calculation
compared to Proposition 6.32 of \cite{CD18I} in that Proposition \ref{p:veri01} characterizes $v(t, \mu)$ while the latter
does the verification of its kernel $V(t, x, \mu)$.
In this sense,  Proposition \ref{p:veri01} can be considered as a generalization of Proposition 5.108 of \cite{CD18I} with general cost structure.
It is also worth mentioning
that the $\inf_{a\in \mathbb R} H(\cdot)$ is used for the optimality condition \eqref{eq:opt01} to simplify our calculation, but
it can be replaced by $\inf_{a\in \mathcal M(\mathbb R)} \langle H(\cdot), \mu\rangle$
for a general purpose.
As mentioned, our main objective in this paper is to obtain explicit solutions of the control problems.

\section{LQG: Fully Observable Case}
\label{s:lqg}
\subsection{Setup}

We consider the following simplified version of
mean-field LQG problem.
 It appears to be more instructive to choose a simpler formulation
so that we can bring out the main feature of the underlying problem.
For
general setup \eqref{eq:con1}, \eqref{eq:con2}, and \eqref{eq:con3}, the coefficients or the functions are given as
\begin{equation}
 \label{eq:para01}
b(t, x, u) = A_{t} x + B_{t} u,
\ \ell(t, x, u) = Q_{t} u^{2},
\end{equation}
and
\begin{equation}
 \label{eq:para02}\barray
g(\mu_{T})\ad = D_{1} [\mu_{T}]_{2}
+ D_{2} [\mu_{T}]_{1}^{2}\\
\ad
= D_{1} \mathbb E[ X_{T}^{2}]
+ D_{2} (\mathbb E[X_{T}])^{2},\earray
\end{equation}
for some continuous and bounded $A_t, B_t, Q_t$ and
constants $D_1, D_2$.
Note that $g$ is polynomial of degree $2$ in $\mu$.

\begin{example} \label{ex:01}
{\rm
 (A standard LQG.)
If
\begin{equation}
 \label{eq:ex01}
A \equiv 0, B\equiv  1, \sigma \equiv 1, Q \equiv 1, D_{1} = 1, D_{2}  = 0,
\end{equation}
then the problem is a standard LQG problem. Note that terminal cost $g(\mu_T) = [\mu_T]_2$ is linear in measure.
In this case, the dynamic programming principle
is applicable
and its HJB can be explicitly solved.
}
\end{example}

\begin{example}
\label{ex:02}
{\rm
This problem is taken from \cite{Yon15}.
Let
\begin{equation}
 \label{eq:ex01}
A \equiv 0, B \equiv  1,\sigma \equiv 1, Q \equiv 1, D_{2} = 1, D_{1}  = 0.
\end{equation}
Note that,
the terminal cost $g(\mu_T) = [\mu_T]_1^2$ is  a
quadratic function in $\mu_{T}$ and  the HJB does not hold.
}
\end{example}

\subsection{Semi-Explicit Solution in Terms of Riccati Equations}

In this section, we solve explicitly the master equation
\eqref{eq:hjb02}-\eqref{eq:tc02} and  apply Proposition \ref{p:veri01}
to  the control problem.
\begin{enumerate}
 \item[(A1)] $Q_{t} >0$ for all $t$.
\end{enumerate}

With  parameters given by \eqref{eq:para01},
the Hamiltonian in the
optimality condition \eqref{eq:opt01}
is  quadratic  in action $a$,
$$
H(t, x, \mu, v, a) =
(A_{t} x + B_{t} a) \partial_{\mu} v (t, \mu, x) +
Q_{t} a^{2}.
$$
Since $Q_{t} >0$,
the infimum over $a \in \mathbb R$ is attained at
$$
a^{*}(t, \mu, x)  = - \frac{ B_{t} \partial_{\mu} v (t, \mu, x)}{2 Q_{t}}
$$
with its minimum
$$
\inf_{a\in\mathbb R} H(t, x, \mu, v, a)  = A_{t} x \partial_{\mu} v - \frac{ B_{t}^{2} }{4 Q_{t}} |\partial_{\mu} v |^{2}.
$$
Therefore,
master equation \eqref{eq:hjb02} becomes
\begin{equation}
 \label{eq:hjb03}
\langle
L_{0} v (t, \mu, \cdot),
\mu
\rangle
+ \partial_{t} v (t, \mu)
= 0,
\end{equation}
where the operator $L_{0}$ is defined by
$$
L_{0} v :=
\Big( A_{t} x \partial_{\mu} v  - \frac{ B_{t}^{2} }{4 Q_{t}} |\partial_{\mu} v |^{2}
+ \frac 1 2 \sigma_t^{2} \partial_{x\mu} v \Big).
$$
Similar to
the traditional approach in LQG, we start with a guess of the value function in
a quadratic function form
$$
v (t, \mu) =
\phi_{1}(t) [\mu]_{2} +
\phi_{2}(t) [\mu]_{1}^{2}
+ \phi_{3}(t).
$$
Then we use the method of un-determined
 ``coefficients'' to determine the three dimensional vector function
$\phi = (\phi_{1}, \phi_{2}, \phi_{3})$.
One can directly write the derivative as
$$
\partial_{\mu} v (t, \mu, x) =
2 \phi_{1} (t) x +
 2 \phi_{2}(t) [\mu]_{1},
$$
which is a polynomial in $x$.
 Moreover, we have
$$
\partial_{t} v (t, \mu) =
\phi'_{1}(t) [\mu]_{2} +
\phi'_{2}(t) [\mu]_{1}^{2}
 + \phi'_{3}(t),
$$
and
$$
\partial_{x \mu} v (t, \mu, x) = 2 \phi_{1} (t).
$$
By plugging  the derivatives in
\eqref{eq:hjb03} and combining the like terms, the
master equation yields that
\begin{equation}
 \label{eq:hjb04}
0 = [\mu]_{2} L_{1} \phi(t) +
[\mu]_{1}^{2} L_{2} \phi(t) +
L_{3} \phi(t),
 \end{equation}
where
$L = [L_{1}, L_{2}, L_{3}]: C^{1}((0, T), \mathbb R^{3})
\mapsto C((0, T), \mathbb R^{3})$
are operators acted on the vector function
$\phi = (\phi_{1}, \phi_{2}, \phi_{3})$ as
$$\begin{array}
 {ll}
L_{1} \phi(t) &=
\displaystyle
\phi_{1}' (t)- \frac{B_{t}^{2}}{Q_{t}} \phi_{1}^{2}(t) + 2 A_{t} \phi_{1}(t), \\
\displaystyle
L_{2} \phi(t) &=
\displaystyle
\phi_{2}'(t) - \frac{B_{t}^{2}}{Q_{t}} \phi_{2}^{2}(t)
- \frac{2 B_{t}^{2}}{Q_{t}} \phi_{1} (t) \phi_{2}(t) + 2A_{t} \phi_{2}(t),\\
\displaystyle
L_{3} \phi(t) &=
\displaystyle \phi_{3}'(t) + \sigma_{t}^{2} \phi_{1}(t).
\end{array}
$$
Since \eqref{eq:hjb04} holds for all $\mu$ together with terminal condition, we have the following system of ODEs in terms of the first-order differential operator $L$
\begin{equation}
 \label{eq:Riccati}
 L\phi (t)= 0, \forall t\in (0, T), \hbox{ with } \phi(T) = (D_{1}, D_{2}, 0).
\end{equation}
Note that $L\phi$ is a linear combination of $\phi'(\cdot)$ and quadratic functions in $\phi$. Such a system $L\phi = 0$ is referred to as a system of  Riccati equations.
One can easily verify the growth condition for $\partial_{x\mu} v$,
and
carry out
verification theorem to conclude the following result.
Furthermore, one can readily verify that the optimal path follows Gaussian process. 

\begin{theorem}
 \label{t:Riccati}
 Suppose $Q_{t} >0$ for all $t$, and
 there exists $\phi \in C^{1}((0, T), \mathbb R^{3})$ for
 Riccati system \eqref{eq:Riccati}. Then the pair $(v, a^{*})$ given by
 $$
v (t, \mu) =
\phi_{1}(t) [\mu]_{2} +
\phi_{2}(t) [\mu]_{1}^{2}
+ \phi_{3}(t),
$$
and
$$
a^{*}(t, \mu, x) =
 - \frac{B_{t}}{Q_{t}} ( \phi_{1} (t) x +
 \phi_{2}(t) [\mu]_{1})
$$
solves the master equation
\eqref{eq:hjb02}-\eqref{eq:tc02} and
the optimality condition \eqref{eq:opt01}.
Moreover, if
$J(u^{*})$ of \eqref{eq:con2}
with parameter sets
\eqref{eq:para01}-\eqref{eq:para02}
is well defined via $(X^{*}, u^{*})$ satisfying
\eqref{eq:con1}-\eqref{eq:con2} and
$$
u^{*}_{t} = a^{*}(t, \mu^{*}_{t}, X^{*}_{t}),
$$
then $(X^{*}, u^{*})$ are optimal trajectory and optimal control, and the optimal value is
$$V^{*} = v (0, \delta_{x}).$$
\end{theorem}

\subsection{Examples: Explicit Solutions}
We use Theorem \ref{t:Riccati} to solve both traditional LQG Example \ref{ex:01} and mean-field LQG Example \ref{ex:02}. In both cases, the Riccati system \eqref{eq:Riccati} becomes
\begin{equation}
 \label{eq:Riccati02}
\begin{array}
 {ll}
 \phi_{1}' &= \phi_{1}^{2}, \\
 \phi_{2}' &= \phi_{2}^{2} + 2\phi_{1} \phi_{2}, \\
 \phi_{3}' &=  - \phi_{1},
\end{array}
\end{equation}

\subsubsection{Solution of Example \ref{ex:01}}
\label{s:ex01}
This problem can be solved using traditional LQG approach; see \cite{YZ99}.
To use Theorem \ref{t:Riccati}, one can solve \eqref{eq:Riccati02} with terminal condition
$$\phi_{1}(T) = 1, \phi_{2}(T) = \phi_{3}(T) =0.$$
The solution for
this Riccati system can be written as follows. For all  $t \in (0,T)$,
\bea
\ad \phi_{2}(t)  =0, \\
\ad \phi_{1}(t) = \frac{1}{1+T-t}, \\
\ad \phi_{3}(t) = \ln(1 + T -t),
\eea
which yields
the optimal strategy
$$u^{*}_{t} = - \frac{X^{*}_{t}}{1 + T - t},$$
and the value function
$$v(t, \mu) = \frac{[\mu]_{2}}{1+T-t} + \ln (1 + T -t).$$
Thus, the optimal value is
$$V^{*} = v(0, \delta_{x}) =  \frac{x^{2}}{1+ T} + \ln(1 + T).$$

\subsubsection{Solution of Example \ref{ex:02}}
\label{s:ex02}
The solution  given by \cite{Yon15} is attained by decoupling FBSDEs and we recover it using Theorem \ref{t:Riccati}.
We solve the Riccati system  \eqref{eq:Riccati02} but with different terminal conditions
$$\phi_{2}(T) = 1, \phi_{1}(T)  =\phi_{3}(T) =0.$$
The solution for this Riccati system can be written as: For all  $t \in (0,T)$
$$
\phi_{1}(t) =\phi_{3}(t) =0, \hbox{ and }
\phi_{2}(t) = \frac{1}{1+T-t}.
$$
Hence,
the optimal strategy is
$$u^{*}_{t} = - \frac{\mathbb E[X^{*}_{t}]}{1 + T - t}$$
and
the value function is
$$v (t, \mu) = \frac{1}{1+T-t} \Big(\int_{\mathbb R} x \mu(dx) \Big)^{2},$$
which implies the optimal value
$$V^{*} = \frac{x^{2}}{1+ T}.$$

\section{\mv LQG:
Controlled Systems under Partial Observations}
\label{s:cn}

 The following interesting question considered in \cite{WZ19} motivates our second example.
 Given a $\mathbb F = \{\mathcal F_{t}: 0\le t\le T \}$ progressively measurable process $u:[0, T]\times \Omega \mapsto \mathbb R$,
 we say $u\in L_{\mathbb F}^{2}$ if
 $\mathbb E [\int_{0}^{T} |u_{s}|^{2} ds] <\infty.$
A deterministic function $u: [0, T]\mapsto \mathbb R$ is said to be $u\in L^{2}([0, T])$, if
 $\int_{0}^{T} |u_{s}|^{2} ds <\infty$.
 Note that both
 $L_{\mathbb F}^{2}$ and $L^{2}([0, T])$ are both Hilbert spaces.
 We ask the question:
 \begin{itemize}
 \item How does the optimal value of \eqref{eq:con1}-\eqref{eq:con3}
 change if
 $L^2_{\mathbb F}$ is replaced by $L^2([0,T])$?
 \end{itemize}
Roughly speaking, the question can be interpreted as:
What is the
infimum that can be achieved
if the control $u
$ is only allowed to be a deterministic process instead of a random one? It is obvious that the optimal
value achieved in the space of deterministic controls is no less than the value with random controls due to  $L^2([0,T]) \subset L^2_{\mathbb F}$.
In what follows, we consider more general questions.

\subsection{Setup}
Recall that we are working with
$(\Omega, \mathcal F, \mathbb P, \mathbb F)$. Suppose that
 on
this filtered probability space,
  there exist two independent
  Brownian motions $\hat W$ and $\wdt W$, respectively. For simplicity,
  we assume
  $\mathbb F = \mathbb F^{\hat W} \times \mathbb F^{\wdt W}$ and
  $\mathcal F = \mathcal F_{T}^{\hat W} \times \mathcal F_{T}^{\wdt W}$, where $\mathbb F^{\hat W} = (\mathcal F_{t}^{\hat W})_{0\le t\le T}$ and
$\mathbb F^{\wdt W} = (\mathcal F_{t}^{\wdt W})_{0\le t\le T}$ are the filtrations generated by $\hat W$ and $\wdt W$, respectively.

Let $\hat \sigma, \wdt \sigma, \hat \eta, \wdt \eta$ be nonnegative constants satisfying
$$\hat \sigma^{2} + \wdt \sigma^{2} =1,\ \hat \eta^{2} + \wdt \eta^{2} =1.$$
A generic player with its initial state $X_{s}$ at time $s$ has its evolution under control $u$
 in the form of
\begin{equation}
 \label{eq:sde02}
 X_{t} = X_{s} + \int_{s}^{t} u_{r} dr +
 \int_{s}^{t}  \hat \sigma d \hat W_{r}
 + \int_{s}^{t}      \wdt \sigma d\wdt W_{r}.
\end{equation}
For simplicity, we require
$X_{s}$ to have a
normal distribution $\mathcal N(x, s)$ given by
\begin{equation}
 \label{eq:Xs02}
 X_{s} = x +
\hat \eta \hat W_{s} + \wdt \eta \wdt W_{s}.
\end{equation}
The cost functional to be minimized
is given by
\begin{equation}
 \label{eq:J02}
J(u
) =
\mathbb E \Big[ \int_{s}^{T} u_{r}^{2} dr \Big] +
D_{1} [\mu_{T}]_{2}
+ D_{2} [\mu_{T}]_{1}^{2}.
\end{equation}
The distinction of the current problem compared with the previous control problem is the following crucial point. Though the player wants to minimize the cost functional, he or she cannot directly access to the state $X_{t}$ due to the lack of the knowledge for $\wdt W_{t}$ and hence for $W_{t}$. Instead, he or she is up to
design a controller using the prediction process
\begin{equation}
 \label{eq:Xhat02}
\hat X_{t} = \mathbb E[ X_{t} | \mathcal F_{t}^{\hat W}].
\end{equation}
We denote by $\hat \mu_{t}$ the distribution induced by $\hat X_{t}$, i.e.,
$\hat \mu_{t} = \mathbb P\hat X_{t}^{-1}.$ Indeed, $\hat X_t$ can be written as
\begin{equation}
 \label{eq:pred01}
 \hat X_{t} = x + \hat \eta \hat W_{s} +
 \int_{s}^{t} u_{r} dr
 + \int_{s}^{t} \hat \sigma  d \hat W_{r},
\end{equation}

Now we are ready to define the optimal value under partial observation by
\begin{equation}
 \label{eq:v02}
V^{*} = \inf_{u\in \hat{\mathcal U}} J(u
),
\end{equation}
where the control space is defined as
\begin{definition}
 A random process
$u: [0, T]\times \Omega \mapsto \mathbb R$ is  said to be admissible
if $u \in L_{\mathbb F}^2$ together with $(X, J)$ satisfies
\eqref{eq:sde02}-\eqref{eq:J02}  and
\begin{equation}
 \label{eq:con31}
u_{t} = a (t, \hat \mu_{t}, \hat X_{t}) \hbox{ for all } t\in [0, T]
\end{equation}
for some controller $a$.
The collection of all such admissible controls is denoted by
 $\hat{\mathcal U}$. 
\end{definition}

Note that if $u \in L^2([0, T])$, then one can
verify with $a(t, \mu, x) = u_t$ that
$u \in  \hat {\mathcal U}$ by definition.
We remark that if
$s= 0$ and $\hat \sigma = 0$, then $\hat X_t$ of \eqref{eq:pred01} is deterministic,
 $L^2([0, T]) = \hat {\mathcal U}$ holds.

\subsection{Semi-Explicit Solution: Separation Principle}

We use the separation principle in filtering theory. The treatment of the problem is outlined below.

\begin{itemize}
 \item[] Step 1:
 Let $\hat X$ be the prediction of $X$ given by \eqref{eq:Xhat02} and
 $\mathcal E$ and $P$ are the error term and variance of the error term:
$$\mathcal E_{t} = X_{t} - \hat X_{t}, \ P_{t} = \mathbb E [\mathcal E_{t}^{2}].$$
Then, $\mathcal E$, and $P$ satisfy
$$
\mathcal E_{t} = \wdt \eta \wdt W_{s} + \wdt \sigma (\wdt W_{t} - \wdt W_{s}),
$$
and
$$
P_{t} = \wdt \eta^{2} s + \wdt \sigma^{2} (t-s).
$$
Recall that $\hat \mu_{t}$ to denote the distribution of $\hat X_{t}$.
Owing to
$$[\mu_{T}]_{1} = [\hat \mu_{T}]_{1}, \
[\mu_{T}]_{2} = [\hat \mu_{T}]_{2} + P_{T},$$ we can rewrite the cost by
$$
J(u) = \hat J(u) + D_{1} P_{T},
$$
where
\begin{equation}
 \label{eq:cost02}
\hat J(u) =
\mathbb E \Big[ \int_{s}^{T} u_{r}^{2} dr \Big] +
D_{1} [\hat \mu_{T}]_{2}
+ D_{2} [\hat \mu_{T}]_{1}^{2}.
\end{equation}

\item[] Step 2: Since $P_{T}$ is independent to the control $u$,
to minimize $J(u
)$, it is sufficient to minimize $\hat J(u )$.
Next we can apply Theorem \ref{t:Riccati} with parameters
$$
A \equiv 0,\  B \equiv 1, \ \sigma_{t} =  \hat \sigma, \  Q \equiv 1
$$
for
$$\hat V^{*} = \inf_{u  \in \hat {\mathcal U}} \hat J(u )$$
with  $\hat J$ of \eqref{eq:cost02} subject to the process $\hat X$ of \eqref{eq:pred01}. 
This yields the Riccati system
\begin{equation}
 \label{eq:Riccati03}
\begin{array}
 {ll}
  \phi_{1}' = \phi_{1}^{2},  \\
  \phi_{2}' = \phi_{2}^{2} + 2 \phi_{1} \phi_{2}, \\
  \phi_{3}' = - \hat \sigma^{2} \phi_{1}, \\
  \phi_{1}(T) = D_{1}, \ \phi_{2}(T) = D_{2}, \ \phi_{3}(T) = 0.
\end{array}
\end{equation}

\end{itemize}

Now we summarize the result in the following proposition.

\begin{proposition}
 \label{p:cn}
 Suppose $\phi= (\phi_{1}, \phi_{2}, \phi_{3}) \in C^{1}([0, T], \mathbb R^{3})$ solves Riccati system \eqref{eq:Riccati03}. Then,
 the optimal strategy for the control problem \eqref{eq:v02} is
 $$u^{*}_t =
 - \phi_{1}(t) \hat X^{*}_{t} -
 \phi_{2}(t) \mathbb E[\hat X^{*}_{t}], \ \forall t\in (s, T),$$
 and the value is
 \begin{equation}
 \begin{array}
 {ll}
 V^{*}
& = \phi_{1}(s) (x^{2} + \hat \eta^{2} s) +
 \phi_{2}(s) x^{2}\\
 & \quad + \phi_{3}(s) + D_{1}  (\wdt \eta^{2} s + \wdt \sigma^{2} (T - s)).
 \end{array}
 \end{equation}
\end{proposition}

\begin{proof}
By Theorem \ref{t:Riccati}, the solution of the master equation
$\hat v^{*}$ and the optimized controller $\hat a^{*}$ associated to
$\hat J$  of \eqref{eq:cost02} and the state
prediction $\hat X$ of \eqref{eq:pred01}  are given by
$$
 \hat v^{*}(t,  \hat \mu) = \phi_{1}(t) [\hat \mu]_{2} +
 \phi_{2}(t) [\hat \mu]_{1}^{2} + \phi_{3}(t),
$$
 and
 $$\hat a^{*} (t, \hat \mu, \hat x) = - \phi_{1}(t) \hat x -
 \phi_{2}(t) [\hat \mu]_{1}.$$
 Moreover, the strategy
 \bea u^{*}_t\ad = \hat a^{*}(t, \hat \mu_{t}, \hat X^{*}_{t})\\
 \ad =
 - \phi_{1}(t) \hat X^{*}_{t} - \phi_{2}(t) \mathbb E[\hat X^{*}_{t}], \ \forall t\in (s, T)\eea
 makes $\hat X^*$ of \eqref{eq:pred01} well defined as a Gaussian process. So, $u^{*} $ given above is optimal and the corresponding value for \eqref{eq:cost02} is given by $\hat V^{*} = \hat v^{*}(s, \hat \mu_{s})$, and finally the value of \eqref{eq:v02} is
 $$  V^{*} = \hat V^{*} + D_{1} P_{T},$$
 which yields the desired conclusion.
  \end{proof}

\subsection{Two Examples}
\begin{example}
\label{ex:03} (linear terminal cost in measure)
With $(D_1, D_2) = (1, 0)$, we
solve the optimization of \eqref{eq:v02} defined through partially observed system
\eqref{eq:sde02}, \eqref{eq:Xs02}, \eqref{eq:J02}. Solving the Riccati system \eqref{eq:Riccati03}, we have
\bea \ad \phi_{1}(t) = \frac{1}{1+T-t}, \\
\ad \phi_{2} \equiv 0, \\
\ad \phi_{3}(t) = \hat \sigma^{2} \ln (1+ T-t).\eea
Then,
 the optimal strategy is
 $$u^{*}_t =
 - \frac{\hat X^{*}_{t}}{1+T-t}, \ \forall t\in (s, T)$$
 and the value is
 $$
 \begin{array}
 {ll}
 V^{*} & =  \frac{1}{1+T-s} (x^{2} + \hat \eta^{2} s)
 \\
 & \hspace{.2in}
  + \hat \sigma^{2} \ln (1+ T-s)
 + \wdt \eta^{2} s + \wdt \sigma^{2} (T - s).
 \end{array}
 $$
 It is noted that the above value with $s=0$ is
 $$V^{*}\Big|_{s = 0} = \frac {x^{2}} {1+T} + \hat \sigma^{2} \ln (1+T)
  + \wdt \sigma^{2} T,$$
Moreover, if $\hat \sigma = 1$ and $\wdt \sigma = 0$, then the above value recovers  the solution of fully observable traditional LQG;
see Example \ref{ex:01} in Section \ref{s:ex01}.
\end{example}

\begin{example}
\label{ex:04} (quadratic terminal cost in measure)
With $(D_1, D_2) = (0, 1)$,
we solve the optimization of \eqref{eq:v02} defined through \eqref{eq:sde02}, \eqref{eq:Xs02}, \eqref{eq:J02}. Solving the Riccati system \eqref{eq:Riccati03}, we have
\bea \ad\phi_{2}(t) = \frac{1}{1+T-t}, \\  \ad\phi_{1} \equiv 0, \\ \ad\phi_{3} \equiv 0.\eea
Then, the optimal strategy is given by
 $$u^{*}_t =
  - \frac{\mathbb E[\hat X^{*}_{t}]}{1+T-t}, \ \forall t\in (s, T)$$
 and the value is
 $$
V^{*} =  \frac{x^{2}}{1+T-s}.
 $$
 Note that
 the above value with $s = 0$ and $\hat \sigma = \hat  \eta = 1$ recovers
 the solution of fully observable mean field LQG; see
 Example \ref{ex:02} of Section \ref{s:ex02}.
 Interestingly, the value is invariant with respect to the observability, i.e., $\partial_{\hat \sigma} V^{*} = 0$.
\end{example}

The computations above both agree with our intuition;
the value
is non-increasing with respect to $\hat \sigma$.
Interestingly, as $\hat \sigma$ increases,
the value is strictly decreasing for  Example \ref{ex:03},
while stays constant for Example \ref{ex:04}.
With that being said, observation of
the noise does not help in minimization for the proper quadratic terminal cost.

\section{Summary}\label{s:sum}
This paper focuses on
mean-field LQGs
with some examples.
These simplified frameworks make it possible
to obtain some explicit solutions that provide us with valuable insight
to a potentially complicated system.
For instance,
Proposition \ref{p:cn} along with Example \ref{ex:03} and \ref{ex:04} clearly
indicates that the value function of a partially observable system depends
not only on the distribution $\mu_s$ of the initial state $X_s$, but on its
joint distribution of $(\hat X_s, X_s- \hat X_s)$
in the observable probability space and its orthogonal probability space. Thus, to characterize the value function in the form of $V(t, \mu)$ depending only on the time and initial distribution  is not sufficient (cf. (4.7) in \cite{PW17}).

The result can be extended to
multidimensional problems with
no essential difficulty but
 more complex notation. For instance, we consider the process $X_t \in \mathbb R^d$ and the cost
 given by
$$\begin{array}
{ll}
dX_t =(A_t X_t + B_t u_t) dt + \sigma_t dW_t,
\\
J(u) =\mathbb{E}\Big[ \int_{0}^{T}  u_t^\top Q_t u_t dt \Big] + g(\mu_{T})
\end{array}
$$
with
$$
g(\mu_{T})=  \int_{\mathbb{R}^{d}} x^\top D_{1}x \mu_{T}(dx)+
[\mu_T]_1^\top D_{2} [\mu_{T}]_{1}.
$$
Solving the master equation yields the following Riccati system:
$$
\begin{array}
{ll}
\phi^{'}_{1}(t)-\phi^{\top}_{1}(t)B_{t}Q_{t}^{-1}B_{t}^{\top}\phi_{1}(t)+2A_{t}^{\top}\phi_{1} (t)=0,
\\
\phi^{'}_{2}(t)-2\phi^{\top}_{2}(t)B_{t}Q_{t}^{-1}B_{t}^{\top}\phi_{1}(t)
-
\\ \hspace{.5in}
\phi^{\top}_{2}(t)B_{t}Q_{t}^{-1}B_{t}^{\top}\phi_{2}(t)+2A_{t}^{\top}\phi_{2}(t)=0,
\\
\phi^{'}_{3}(t)+tr[\sigma_{t}\sigma_{t}^{\top}\phi_{1}(t)]=0,
\end{array}
$$
with the terminal condition
$$
\phi_{1}(T)=D_{1}, \phi_{2}(T)=D_{2}, \phi_{3}(T)=0.
$$

More challenging generalization is to consider
more general
cost.
For instance, going back to 1-d problem  \eqref{eq:con1}, \eqref{eq:con2},  \eqref{eq:con3}, \eqref{eq:para01} with terminal cost
$$g(\mu_T) = \mathbb E[X_T^2] +
 (\mathbb E[\Psi(X_{T})])^{2},$$
one shall solve the master equation with a guess
$$v = \phi_1 \langle \psi, \mu\rangle^2 + \phi_2 \langle x^2, \mu\rangle + \phi_3 +
\phi_4 \langle \psi, \mu\rangle \langle x, \mu \rangle + \phi_5 \langle x, \mu\rangle^2.
$$


\balance

\def\cprime{$'$}

\end{document}